\documentclass[runningheads]{llncs}
\usepackage{graphicx}
\usepackage{enumerate}
\usepackage{xcolor,amssymb,amsmath}

\usepackage{hyperref}
\hypersetup{
    colorlinks=true,
    linkcolor=blue,
    filecolor=blue,
    urlcolor=blue,
    citecolor=blue,
    filecolor=blue
    }

\newcommand\hide[1]\empty

\newcommand\todo[1]{ [~ {\color{red} #1 }]}
\renewcommand\todo[1]\empty

\newcommand\improve[1]{ [~ {\color{green} #1 }]}
\renewcommand\improve[1]\empty
\newcommand\obsolete[1]{{\color{lightgray}\noindent{\bf Obsolete:} {#1}}}
\renewcommand\obsolete[1]\empty
\newcommand\later[1]{{\color{green}\noindent{\bf later:} \textit{#1}}}
\renewcommand\later[1]\empty

\newcommand\C[1]{\chi(#1)}
\newcommand\CF[1]{\chi^{\scriptscriptstyle >}_{#1}}

\newcommand\diff[1]{{#1}_{\neq}}
\newcommand\dunr[1]{{#1}^{(\neq)}}

\newcommand\framets[1]{#1}

\def\frF{\framets{F}}

\def\mM{\framets{M}}

\def\clF{\mathcal{F}}
\def\clD{\mathcal{D}}
\def\clG{\mathcal{G}}
\def\Log{\myoper{Log}}
\newcommand\myoper[1]{\mathop{\myopts{#1}}}
\newcommand\myopts[1]{\mathrm{#1}}

\def\Di{\lozenge}

\def\imp{\rightarrow}

\newcommand\logicts[1]{{\textsc{#1}}}

\def\vL{L}

\def\clA{\mathcal{A}}

\def\clB{\mathcal{B}}
\def\clC{\mathcal{C}}
\def\clD{\mathcal{D}}

\def\clP{\mathcal{P}}
\def\EE{\exists}
\def\AA{\forall}

\def\v{\theta}

\def\vf{\varphi}
\def\mo{\vDash}
\def\vd{\vdash}
\def\con{\wedge}

\def\emp{\varnothing}
\def\SubFrs{\myoper{Sub}}

\def\se{\subseteq}
\newcommand\languagets[1]{\logicts{#1}}

\def\PV{\languagets{PV}}
\def\wK4{\logicts{wK4}}

\def\LB{\LK{B}}

\newcommand{\ff}[1]{\widehat{#1}}
\newcommand\scheme[1]{ {[ #1 ]}}
\newcommand{\LK}[1]{\logicts{K#1}}
\newcommand\lng[1]{{\ell(#1)}}
\def\ACon{\logicts{Con}}
\def\PF{\logicts{P}}
\def\Did{\langle\neq\rangle}
\def\Boxd{ {[ \neq ]} }

\begin{document}
\title{Decidability of modal logics of non-$k$-colorable graphs}

%
%
\author{Ilya Shapirovsky\inst{1}\orcidID{0000-0001-7434-5894}}
\authorrunning{Ilya Shapirovsky}
%
\institute{New Mexico State University, USA\\
\email{ilshapir@nmsu.edu}}
\maketitle              
\begin{abstract}
We consider the bimodal language, where the first modality is interpreted
by a binary relation in the standard way, and the second is interpreted by the relation of inequality.
It follows from Hughes (1990), that in this language, non-$k$-colorability of a graph is expressible for every finite $k$. We show that modal logics of classes of non-$k$-colorable graphs (directed or non-directed), and some of their extensions, are decidable.

\end{abstract}
\section{Introduction}

It is known that a non-$k$-colorability of a graph can be expressed by
propositional modal formulas \cite{Hughes1990}. In \cite{GHV04}, such formulas were used
to construct a canonical logic which cannot be determined by a first-order definable class of relational structures; this gave a solution
of a long-standing problem by Fine \cite{FineCanonicalFO1975}.

In this paper, we are interested in decidability of modal logics given by axioms of non-$k$-colorability, and some of their extensions.  We consider the bimodal language, where the first modality is interpreted
by a binary relation in the standard way, and the second (difference modality) is interpreted by the
relation of inequality.

The paper has the following structure. Section \ref{sec:prel} provides preliminary
syntactic and semantic facts.  In Section \ref{sec:chrom},  the finite model property and decidability are shown for logics of non-$k$-colorable graphs. In Section \ref{sec:conn}, these results are obtained for the  connected non-directed case. Further results on the finite model property
of logics of non-$k$-colorable graphs
are obtained in Section \ref{sec:coroll}. A discussion is given in Section \ref{sec:concl}.

\section{Preliminaries}\label{sec:prel}

We assume that the reader is familiar with basic notions in modal logic (see, e.g., \cite{CZ,BRV-ML} for the references). Below we briefly remind some of them.\later{redo}

\subsubsection{Modal syntax and relational semantics.}
The set of {\em $n$-modal formulas} is built from
a countable set of {\em variables} $\PV=\{p_0,p_1,\ldots\}$ using Boolean connectives $\bot,\imp$ and unary connectives $\Di_i$, $i< n$ ({\em modalities}).
Other logical connectives
are defined as abbreviations in the standard way, in particular $\Box_i\vf$
denotes $\neg \Di_i \neg \vf$.

An {\em $n$-frame} is a structure  $\frF=(X,(R_i)_{i<n})$,
where $X$ is a non-empty set and ${R_i\se X{\times}X}$ for $i<n$.
A {\em valuation in a frame $F$} is a map $\PV\to \clP(X)$,
where $\clP(X)$ is the set of all subsets of $X$.
A {\em (Kripke) model on} $\frF$ is a pair $(\frF,\theta)$, where $\theta$ is a valuation.
The \emph{truth} of formulas in models is defined in the usual way:
\begin{itemize}
    \item $M, x \models p_i$ iff $x \in \v(p_i)$;
    \item $M, x \not\models \bot$;
    \item $M, x \models \varphi \to \psi$ iff $M, x \models \varphi$ implies $M, x \models \psi$;
    \item $M, x \models \Di_i \varphi$ iff there exists $y$ such that $xR_i y$ and  $M, y \models \varphi$.
\end{itemize}
\todo{Either ??}
A formula $\vf$ is {\em true in a model $\mM$}, in symbols $M\mo\vf$, if $\mM,x\mo\vf$ for all $x$ in $\mM$.
A formula $\vf$ is {\em valid in a frame $\frF$}, in symbols $\frF\mo\vf$,
if $\vf$ is true in every model on $\frF$.
For a class $\clC$ of structures (frames or models) and a set of formulas $\Phi$, we write
$\clC\mo\Phi$, if  $S\mo \vf$ for all $S\in\clC$ and $\vf\in \Phi$.

For the standard notions of {\em generated} and {\em point-generated}  {\em subframe} and {\em submodel}, and {\em p-morphism}, we refer the reader to \cite[Section 3.3]{CZ} or
\cite[Sections 2.1 and 3.3]{BRV-ML}.

\subsubsection{Modal logics.}
A ({\em propositional normal $n$-modal}) {\em logic} is a set $\vL$ of $n$-modal formulas
that contains all classical tautologies, the axioms $\neg\Di_i \bot$  and
$\Di_i  (p_0\vee p_1) \imp \Di_i  p_0\vee \Di_i  p_1$ for each $i<n$, and is closed under the rules of modus ponens,
substitution and {\em monotonicity}; the latter means that for each $i<n$, $\vf\imp\psi\in \vL$ implies $\Di_i \vf\imp\Di_i \psi\in \vL$.\footnote{
For this version of the definition of normal modal logic, see, e.g.,
\cite[Remark 4.7]{BRV-ML}.}
We write $\vL\vd \vf$ for $\vf\in L$. For a set $\Phi$ of $n$-modal formulas,
$\vL+\Phi$ is the smallest normal logic containing $\vL\cup\Phi$.
For a formula $\vf$, $L+\vf$ abbreviates $L+\{\vf\}$.
$\LK{}$ denotes the smallest unimodal logic.

An {\em $L$-frame} is a frame where $L$ is valid.

For a class $\clC$ of $n$-frames, the set of $n$-modal formulas
$\vf$ such that $\clC\mo\vf$ is called the {\em logic of $\clC$}
and is denoted by $\Log{\clC}$. It is straightforward that $\Log{\clC}$ is a  normal logic.
Such logics are called {\em Kripke complete}.
A logic has the {\em finite model property} (fmp),
if it is the logic of a class of finite frames (by the cardinality of a frame or model we mean the cardinality of its domain).
We say that $L$ has the {\em exponential fmp}, if
for every formula $\vf\notin L$, $\vf$ is falsified in an $L$-frame of cardinality $\leq 2^\lng{\vf}$,
where $\lng{\vf}$ is the number of subformulas of $\vf$.

The {\em canonical model $M_L=(X_L,(R_{i,L})_{i<n},\v_L)$ of} $\vL$ is built from
maximal $\vL$-consistent sets $X_L$ of $n$-modal formulas; the canonical relations and the valuation are defined in the standard way. Namely, for $\Gamma,\Delta\in X_L$,
put $(\Gamma,\Delta) \in R_{i,L}$, if $\{\Di_i\vf \mid \vf\in \Delta\}\se \Gamma$, and set
$\v_L(p)=\{\Gamma\in X_L\mid p\in \Gamma\}$ for $p\in\PV$.  The following fact is well known, see e.g., \cite[Chapter 4.2]{BRV-ML}.
\begin{proposition}\label{prop:k-canonical-model}[Canonical model theorem]
$\vL\vd\vf$ iff $M_L\mo\vf$.
\end{proposition}

$L$ is {\em canonical}, if $L$ is valid in its {\em canonical frame} $F_L=(X_L,(R_{i,L})_{i<n})$.
A formula $\vf$ is {\em canonical}, if $F_L\mo \vf$ whenever $\vf\in L$.

\begin{proposition}\label{prop:canon-gener}
Let $L$ be a canonical $n$-modal logic. Then for any $n$-modal logic $L'\supseteq L$,
we have $F_{L'}\mo L$.
\end{proposition}
This fact is well known and follows from a simple observation that
$F_{L'}$ is a generated subframe of $F_L$.

\subsubsection{Logics with the difference modality.}
It is known that adding the difference modality allows to
increase the expressive power of propositional
 modal language (see, e.g.,  \cite{DeRijkeDiff}, \cite{GargovGorankoDiff1993} in the relational context, or \cite{Kudinov2014} for topological semantics).
 \todo{Read Goranko: algebraic? topological - more refs}

Is this paper we will consider bimodal ($n=2$) and unimodal ($n=1$) languages.
We write $\Di$ for $\Di_0$, and
$\Did$ for $\Di_1$; likewise for boxes.
We also use abbreviations $\EE\vf$ for $\Did \vf\vee \vf$ and
$\AA\vf$ for $\Boxd \vf\wedge \vf$.

For a unimodal frame $F=(X,R)$, let
$\diff{F}$ be the bimodal frame $(X,R,\neq_X)$,
where $\neq_X$ is the inequality relation on $X$, i.e.,  the set of pairs $(x,y)\in X\times X$ such that $x\neq y$. For a class $\clF$ of frames,  put $\diff{\clF}=\{\diff{F}\mid F\in\clF\}$

For a unimodal logic $L$, let $\diff{L}$
be the smallest bimodal logic that contains $L$ and the following formulas:
\begin{equation}\label{eq:diff-axioms}
p\imp \Boxd\Did  p, \quad \Did \Did p\imp \EE p, \quad
\Di p\imp \EE p.
\end{equation}
Recall that the validity of $p\imp \Boxd\Did  p$ in a frame $(X,R,D)$
expresses that $D$ is symmetric,  the formula $\Did \Did p\imp \EE p$ means that the relation
$D\cup  Id_X$ is transitive ($Id_X$ denotes the diagonal relation on $X$),  and the
formula $\Di p\imp \EE p$ expresses that $R\subseteq D\cup Id_X$; see, e.g., \cite{DeRijkeDiff} for details.

In particular, it follows that we have the following characterization of bimodal
point-generated frames that validate $\diff{\LK{}}$:
\begin{proposition}\label{prop:point-gen-diff}
$F=(X,R,D)$ is a point-generated $\diff{\LK{}}$-frame iff $\;{\neq}_X \,\subseteq \,D$.
\end{proposition}
The formulas  \eqref{eq:diff-axioms} are Sahlqvist formulas, and hence are canonical (see, e.g.,  \cite[Theorem 10.30]{CZ}).
In particular, it follows that $\diff{\LK{}}$ is Kripke complete. It is well-known that this logic has the finite model property: for every non-theorem $\vf$ of $\diff{\LK{}}$, consider a submodel $M$
of the canonical model of $\diff{\LK{}}$ generated by a point $x$ where $\vf$ is refuted, and take a filtration of $M$.
\begin{proposition}[\cite{DeRijkeDiff}]\label{prop:DeRijkeDiff}
$\diff{\LK{}}$ is the logic of the class of all (finite) frames of the form $(X,R,\neq_X)$.
\end{proposition}
This proposition follows from Proposition \ref{prop:point-gen-diff} and the following standard move that ``repairs'' $D$-reflexive points.
For a point-generated $\diff{\LK{}}$-frame $F=(X,R,D)$, let $\dunr{F}$ be the frame $(Y,S,\neq_Y)$, where
\begin{eqnarray*}
   &&Y \; =\; \{(x,0):x\in X\}\cup \{(x,1):x\in X\,\&\,xDx\}, \\
  &&(x,i)S(y,j) \;\text{ iff }\; xRy.
\end{eqnarray*}
Let $f:X\to Y$ be the map defined by $f(x,i)=x$. Readily, $f$ is a p-morphism from $\dunr{F}$ onto $F$.
Now Proposition \ref{prop:DeRijkeDiff} follows from the p-morphism lemma (see, e.g., \cite[Theorem 3.14(i)]{BRV-ML}).

The frame $\dunr{F}$  will be used later; we will call it the {\em repairing of} $F$.

\section{Logics of non-$k$-colorable graphs}\label{sec:chrom}
By a {\em graph} we mean a unimodal frame $(X,R)$ in which $R$ is symmetric.
A {\em directed graph} is a unimodal frame.
As usual, a  {\em partition} $\clA$ of a set $X$ is a family
of non-empty pairwise disjoint sets such that $X=\bigcup \clA$.

\begin{definition}\normalfont
Let $X$ be a set, $R\subseteq X\times X$.
A partition $\clA$ of $X$
is {\em proper}, if $\AA A\in\clA \,\AA x\in A\,\AA y \in A \;\neg xRy$. Let $$C(X,R)=\{|\clA|:\clA\text{ is a finite proper partition of }X\}.$$
Let $\C{X,R}$ be the least $k$ in $C(X,R)$, if $C(X,R)\neq \emp$,
and $\infty$ otherwise.

In the case when $R$ is symmetric, $\C{X,R}$ is called the {\em chromatic number of the graph} $(X,R)$.
\later{Refs}
\later{$X=\emp$}
\end{definition}

Put
$$\CF{k}=\AA \bigvee_{i<k} (p_i\wedge \bigwedge_{i\neq j<k} \neg p_j) \imp \EE \bigvee_{i<k} (p_i\wedge \Di p_i).$$

\begin{proposition}[\cite{Hughes1990,GHV04}]\label{prop:chrom}
Let $F=(X,R,D)$ be a point-generated $\diff{\LK{}}$-frame.
Then $\C{X,R}>k$ iff $F\mo \CF{k}$.
\end{proposition}
\begin{remark}
Formulas considered in \cite{Hughes1990,GHV04} are formally different.
\end{remark}
\begin{proof}
The premise of $\CF{k}$ says that non-empty values of $p_i$'s form a partition of $X$,
the conclusion says that this partition is not proper. \qed
\end{proof}

In particular, it follows that for every graph $G$,
\begin{equation*}
\text{the chromatic number of }G>k \text{ iff }\diff{G}\mo \CF{k}.
\end{equation*}

To show that logics of non-$k$-colorable graphs have the finite model property, we will use filtrations.

For a model $M = (X, (R_i)_{i<n}, \v)$ and a set of $n$-modal formulas $\Gamma$, put
\begin{center}
$x\sim_{\Gamma} y$ iff $\forall \psi \in\Gamma$ ($M,x\models \psi$ iff $M,y\models \psi$).
\end{center}

For a formula $\vf$, let $\SubFrs{\vf}$ be the set of all subformulas of $\varphi$. A set $\Gamma$ of formulas is {\em $\SubFrs$-closed}, if $\SubFrs{\varphi} \subseteq \Gamma$ whenever $\varphi \in \Gamma$.
\begin{definition}
\normalfont
Let $\Gamma$ be a $\SubFrs$-closed
set of formulas.
A {\em $\Gamma$-filtration} of a model $M = (X, (R_i)_{i<n}, \v)$   is a model $\ff{M}=(\ff{X},(\ff{R}_i)_{i<n},\ff{\theta})$
such that
\begin{enumerate}
\item $\ff{X}=X{/}{\sim}$ for some equivalence relation $\sim$ such that $\sim \;\subseteq \;\sim_\Gamma$;
\item ${\ff{M},[x]\models p}$ iff ${M,x\models p}$ for all $p\in \Gamma$.
Here $[x]$ is the $\sim$-class of $x$.
\item For all $i<n$, we have ${(R_i)}_{\sim} \subseteq \ff{R}_i \subseteq  {(R_i)}_{\sim}^\Gamma$, where
$$
\begin{array}{ccl}
~[x]\,{(R_i)}_\sim\,[y] & \text{iff} & \exists x'\sim x\ \exists y'\sim
y\;
(x'\,R_i\,y'),
\smallskip \\
~[x]\,{(R_i)}_{\sim}^\Gamma\,[y] & \text{iff} & \forall \psi\;
   (\Di_i \psi\in \Gamma \: \& \: M,y\models\psi \Rightarrow M,x\models \Di_i\psi ).
\end{array}
$$
\end{enumerate}
The relations ${(R_i)}_\sim$ are called the \emph{minimal filtered relations}.

If $\sim\; =\; \sim_\Psi$ for some finite  set of formulas ${\Psi\supseteq\Gamma}$,
then $\ff{M}$ is called a \emph{definable $\Gamma$-filtration} of the model~$M$.
\end{definition}

The following fact is well known, see, e.g., \cite{CZ}:
\begin{proposition}[Filtration lemma]\label{prop:Filtration}
Suppose that $\Gamma$ is a finite $\operatorname{Sub}$-closed set of formulas
and $\ff{M}$ is a $\Gamma$-filtration of a model~$M$.
Then, for all points $x$ in $M$ and all formulas ${\varphi \in\Gamma}$,
we have: \ \
\begin{center}
${M,x\models \varphi}$ iff ${\widehat{M},[x]\models\varphi}$.
\end{center}
\end{proposition}

For a bimodal formula $\vf$, let $\scheme{\vf}$ be the set of bimodal formulas that are substitution instances of $\vf$
(the axiom scheme).

\begin{lemma}\label{lem:filtr-for-chrom}
Let $M=(X,R,D,\v)$ be a bimodal model, $k<\omega$, $M\mo\scheme{\CF{k}}$, and let
$\Gamma$ be a finite $\SubFrs$-closed  set of bimodal formulas. Then for every finite $\Psi\supseteq \Gamma$, for every
$\Gamma$-filtration $\ff{M}=(X/{\sim_\Psi},\ff{R},\ff{D},\ff{\v})$ of $M$, we have $\C{X/{\sim_\Psi},\ff{R}}>k$.
\end{lemma}
\begin{remark}
We do not make the assumption that $(X,R,D)$ is a $\diff{\LK{}}$-frame or even that
$M\mo \diff{\LK{}}$.
We also do not assume that $\C{X,R}>k$: in general, $M\mo\scheme{\CF{k}}$ is a weaker condition.
\end{remark}
\begin{proof}
Let $\ff{X}=X/{\sim_\Psi}$.
Since $\Psi$ is finite, for every $A\in \ff{X}$ there is a modal formula $\psi_A$ such
that
\begin{equation}\label{eq:psi-A}
  M,x\mo \psi_A \text{ iff } x\in A.
\end{equation}
Hence, for every $B\subseteq \ff{X}$, for the formula $\vf_B=\bigvee_{A\in B}\psi_A$ we have:
\begin{equation}\label{eq:vf-B}
  M,x\mo \vf_B \text{ iff } x\in \bigcup B.
\end{equation}
We say that $\vf_B$ {\em defines} $B$.

Let $\clB$ be a partition of $\ff{X}$ and $|\clB|=n\leq k$.
Then $\{ \bigcup B : B\in \clB\}$ is a partition of $X$.
Let $\vf_0,\ldots \vf_{n-1}$ be formulas that define elements of $\clB$. For $n-1<i<k$, let $\vf_i=\bot$.
By \eqref{eq:vf-B}, we have
\begin{equation*}
M\mo \AA \bigvee_{i<k}(\vf_i\con \bigwedge_{i\neq j<k}\neg \vf_j).
\end{equation*}
The result of substitution of $\vf_i$'s for $p_i$'s in $\CF{k}$ is true in $M$, so\later{engl}
\begin{equation*}
M\mo \EE \bigvee_{i<k}(\vf_i\con \Di\vf_i).
\end{equation*}
It follows from \eqref{eq:vf-B} that for some $i$, for some $x,y\in \bigcup B_i$ we have $xRy$.
Let $[x]_\Psi$  denote the $\sim_\Psi$-class of $x$.
We have $[x]_\Psi, [y]_\Psi\in B_i$.
Since $\ff{R}$ contains the minimal filtered relation, $[x]_\Psi \ff{R} [y]_\Psi$.
So $\clB$ is not a proper partition of $(\ff{X},\ff{R})$. \qed
\end{proof}

\obsolete{
It follows that all logics $\LK{}_2+\CF{k}$ have the finite model property and are decidable. Our goal is to transfer this result to the logics of graphs. \todo{improve}
}

Recall that the modal formula $p\imp \Box\Di p$ expresses the symmetry of a binary relation.
Let $\LB$ be the smallest unimodal logic containing this formula. It is well known that
this logic is canonical.

\begin{theorem}\label{thm:fmpCn}
For each $k<\omega$, the logics
$\diff{\LK{}}+\CF{k}$ and $\diff{\LB}+\CF{k}$ have the exponential finite model property and are decidable.
\end{theorem}
\begin{proof}
Let $M_1=(X_1,R_1,D_1,\v_1)$ and $M_2=(X_2,R_2,D_2,\v_2)$ be the canonical models of the logics $\diff{\LK{}}+\CF{k}$ and $\diff{\LB}+\CF{k}$, respectively.
By Proposition \ref{prop:canon-gener}, the canonical frames $(X_1,R_1,D_1)$ and $(X_2,R_2,D_2)$  validate the logic $\diff{\LK{}}$,
and also $R_2$ is symmetric.

Let $L$ be one of these logics, $\vf\notin L$.
Then $\vf$ is false at a point $x$ in the canonical model of $L$. Let $M=(Y,R,D,\v)$ be its submodel generated by $x$. By Proposition \ref{prop:point-gen-diff},  for all $y,z\in Y$ we have:
\begin{equation}\label{eq:diff}
  \text{if } y\neq z, \text{ then } yDz.
\end{equation}

Let $\Gamma=\SubFrs{\vf}$, $\sim\;=\;\sim_\Gamma$.
Put $\ff{Y}=Y{/}{\sim}$, and
consider the filtration $\ff{M}=(\ff{Y}, R_\sim, D_\sim,\ff{\v})$.
Clearly, the size of $\ff{Y}$ is bounded by $2^\lng{\vf}$

By Filtration lemma (Proposition \ref{prop:Filtration}), $\vf$ is falsified in $\ff{M}$.
Let us show that the frame $(\ff{Y}, R_\sim, D_\sim)$ validates $L$.

From \eqref{eq:diff}, it follows that $(\ff{Y}, R_\sim, D_\sim)$ validates the logic $\diff{K}$.
In the case of symmetric $R$, the minimal filtered relation $R_\sim$ is also symmetric.
Finally, by Lemma \ref{lem:filtr-for-chrom}, $\C{\ff{Y}, R_\sim}>k$.
By Proposition \ref{prop:chrom}, $(\ff{Y}, R_\sim, D_\sim)$ validates $L$.

Hence $L$ is complete with respect to its finite frames.
\qed
\end{proof}

\begin{theorem}\label{thm:chrom-compl}
Let $\clG^{>k}$ be the class of graphs $G$ such that $\C{G}>k$,
and let $\clD^{>k}$ be the class of directed graphs $G$ such that $\C{G}>k$.
Then $\Log{\diff{\clG}^{>k}}=\diff{\LB{}}+\CF{k}$,
and $\Log{\diff{\clD}^{>k}}=\diff{\LK{}}+\CF{k}$.
\end{theorem}
\begin{proof}
By Theorem \ref{thm:fmpCn}, the logics
$\diff{\LK{}}+\CF{k}$ and $\diff{\LB}+\CF{k}$  are complete with respect to their finite point-generated frames.

Consider a point-generated $\diff{\LK{}}$-frame $F=(X,R,D)$ and its repairing  $\dunr{F}=(Y,S,\neq_Y)$.
Recall that $F$ is a p-morphic image of $\dunr{F}$.
Let $\clA$ be a partition of $Y$, $|\clA|\leq k$. Consider the following partition $\clB$ of $X$:
$B\in\clB$ iff there is $A\in \clA$ such that $B=\{x: (x,0) \in A\}$ and $B\neq \emp$.

Assume that $\C{X,R}>k$. It follows that for some $B\in \clB$ and  some $x,y\in B$ we have $xRy$.
Then for some $A\in \clA$ we have $(x,0),(y,0)\in A$ and
$(x,0)S (y,0)$. Thus, $\clA$ is not a proper partition of $(Y,S)$. Hence, $\C{Y,S}>k$.
This completes the proof in the directed case: $\Log{\diff{\clD}^{>k}}=\diff{\LK{}}+\CF{k}$.

Clearly, if $R$ is symmetric, then $S$ is symmetric is well. This observation completes the proof in the non-directed case.
\qed
\end{proof}

\begin{remark}
These theorems can be extended for the case of graphs where the relation is irreflexive, if
instead of the formula $\Di p \imp \EE p$ in the definition of $\diff{L}$ we use the formula $\Di p \imp \Did  p$.
Then in any frame $(X,R,D)$ validating this version of $\diff{L}$, the second relation contains $R$,
and so if a point is $R$-reflexive, it is also $D$-reflexive.
In this case, the repairing $\dunr{F}$ should be modified in the following way:
\begin{eqnarray*}
   &&Y \; =\; \{(x,0):x\in X\}\cup \{(x,i):x\in X\,\&\,xDx\,\&\,0<i\leq k\}, \\
  &&(x,i)S(y,j) \;\text{ iff }\; xRy\;\&\;( (x,i)\neq (y,j)).
\end{eqnarray*}
Then $S$ is irreflexive, the map $(x,i)\mapsto x$ remains a p-morphism,
and $R$-reflexive points in $F$ become cliques of size $>k$. Also,  it follows
that $\C{Y,S}>k$ whenever $\C{X,R}>k$.
\end{remark}

\begin{remark}
A related result was obtained very recently in \cite{Ding2023}:
it was shown that in neighborhood semantics of modal language, the non-k-colorability of hypergraphs is expressible, and the resulting modal systems are decidable as well.\footnote{I am grateful to Gillman Payette for sharing with me this reference after my talk at WoLLIC.} 
\end{remark}

\section{Logics of connected graphs}\label{sec:conn}

A frame $F=(X,R)$ is {\em connected}, if for any points $x,y$ in $X$,
there are
points $x_0=x, x_1, \ldots, x_n=y$ such that for each $i<n$, $x_i R x_{i+1}$ or $x_{i+1} R x_{i}$.
\todo{$n<\omega$}

Let $\ACon$ be the following formula:
\begin{equation}\label{eq:conn}
\EE p \wedge \EE \neg  p \imp \EE (p\wedge \Di \neg p).
\end{equation}

\begin{proposition}\label{prop:conn}
Let $F=(X,R,D)$ be a point-generated $\diff{\LB}$-frame.
Then $(X,R)$ is connected iff $F\mo \ACon$.
\end{proposition}
\begin{proof}
Assume that $(X,R)$ is connected and $M$ is a model on $F$ such that
$\EE p \wedge \EE \neg  p$ is true (at some point) in $M$. Hence
there are points $x,y$ in $M$ such that $M,x\mo p$ and $M,y\mo \neg p$.
Then there are $x_0=x, x_1, \ldots, x_n=y$ such that $x_i R x_{i+1}$ for each $i<n$.
Let $k=\max\{i:M,x_i\mo p\}$. Then $M,x_k\mo p\con \Di \neg p$. Hence $\ACon$ is valid in $F$.

Assume that $(X,R)$ is not connected. Then there are $x,y$ in $X$ such that $(x,y)\notin R^*$,
where $R^*$ is the reflexive transitive closure of $R$. Put $\v(p)=$\mbox{$\{z: (x,z)\in R^*\}$s}.
In the model $M=(F,\v)$, we have $M\mo\EE p \wedge \EE \neg  p $. On the other hand, at every point $z$ in $M$ we have $M,z\mo p\imp \Box p$, so the conclusion of $\ACon$ is not true in $M$. So $\ACon$ is not valid in $F$.
\qed
\end{proof}

In particular, it follows that for every graph $G$,
\begin{equation*}
G \text{ is connected iff } \diff{G}\mo \ACon.
\end{equation*}

\begin{remark}
There are different ways to express connectedness in propositional modal languages \cite{Sheht1990}.
In particular, in the directed case, the connectedness can be expressed by the following modification of
\eqref{eq:conn}:
$$
\EE p \wedge \EE \neg  p \imp \EE (p\wedge \Di \neg p)\vee \EE (\neg p\wedge \Di p);
$$
Following the line of \cite{Sheht1990}, one can modally express the property of a graph
to have at most $n$ connected components for each finite $n$.
\end{remark}

It is known that in many cases, adding axioms of connectedness preserves the finite model property
\cite{Sheht1990,GoldblattHodkinsonTangle}. The following lemma shows that this is the case in our setting as well.

\begin{lemma}\label{lem:filtr-for-conn}
Assume that $(X,R,D)$ is a point-generated $\diff{\LK{B}}$-frame.
Let $M=(X,R,D,\v)$ be a model such that $M\mo\scheme{\ACon}$,  and
let $\Gamma$ be a finite $\SubFrs$-closed set of bimodal formulas. Then for every finite $\Psi\supseteq \Gamma$, for every
$\Gamma$-filtration $\ff{M}=(X/{\sim_\Psi},\ff{R},\ff{D},\ff{\v})$ of $M$,
$(X/{\sim_\Psi},\ff{R})$ is connected.
\end{lemma}

\begin{remark}
Similarly to Lemma \ref{lem:filtr-for-chrom}, connectedness of $(X,R)$ does not follow from $M\mo\scheme{\ACon}$.
\end{remark}
\begin{proof}
Let $\ff{X}=X/{\sim_\Psi}$, $c$ the number of elements in $\ff{X}$.
We recursively define $c$ distinct elements $A_0,\ldots,A_{c-1}$ of $\ff{X}$,
and auxiliary sets $\ff{Y}_n=\{A_0,\ldots,A_{n}\}$, $\ff{R}_n=\ff{R}\cap (\ff{Y}_n\times \ff{Y}_n)$ for $n<c$
such that
\begin{equation}\label{eq:conn-n}
\text{
the restriction $(\ff{Y}_n,\ff{R}_n)$ of $(\ff{X},\ff{R})$ to $\ff{Y}_n$ is connected. 
}
\end{equation}

Let $A_0$ be any element of $\ff{X}$. The frame $(\ff{Y}_0,\ff{R}_0)$ is connected, since it is a singleton.

Assume $0<n<c$ and define $A_n$.
By the same reasoning as in Lemma \ref{lem:filtr-for-chrom},
there is a formula $\vf_n$  such that
\begin{equation}\label{eq:vf-conn}
  M,x\mo \vf_n \text{ iff } x\in A_i \text{ for some }i<n.
\end{equation}
The formula
\begin{equation}\label{eq:vf-sub-con}
\EE \vf_n \wedge \EE \neg  \vf_n \imp \EE (\vf_n\wedge \Di \neg \vf_n).
\end{equation}
is a substitution instance of $\ACon$, so it is true in $M$. Let $V=\bigcup \ff{Y}_{n-1}$.
The set $\ff{Y}_{n-1}$ has $n<c$ elements, so there are points $x,y$ in $X$ such that
$x\in V$, and $y\notin V$. So $M,x\mo \vf_n$ and $M,y\mo \neg \vf_n$.
By Proposition \ref{prop:point-gen-diff},
the premise of \eqref{eq:vf-sub-con} is true in $M$. Hence we have
$M,z\mo\vf_n\wedge \Di \neg \vf_n$ for some $z$ in $M$. Then $z\in V$ and
there exists $u$ in $X\setminus V$ with $zRu$. Since $\ff{R}$ contains the minimal filtered relation,
$[z]_\Psi \ff{R} [u]_\Psi$.
We put $A_n=[u]_\Psi$.
By the hypothesis \eqref{eq:conn-n}, $(\ff{Y}_{n-1},\ff{R}_{n-1})$ is connected, and so
$(\ff{Y}_n,\ff{R}_n)$ is connected as well.

Finally, observe that $(\ff{Y}_{c-1},\ff{R}_{c-1})$ is the frame $(\ff{X},\ff{R})$.
\qed
\end{proof}

\begin{theorem}\label{thm:fmpCnConn}
For each $k<\omega$, the logics $\diff{\LB}+\{\ACon,\Di\top\}$  and
$\diff{\LB}+\{\CF{k}, \ACon,\Di\top \}$ have the exponential finite model property and are decidable.
\end{theorem}
\begin{proof}
Similar to the proof of Theorem \ref{thm:fmpCn}.
Let $\vf$ be a non-theorem of one these logics,
$M$ a point-generated submodel of the canonical model of the logic where $\vf$ is falsified.
Consider the frame $F$ of the minimal filtration of $M$ via the subformulas of $\vf$.
We only need to check that $F$ validates $\Di \top$ and $\ACon$ (validity of other axioms was checked in the proof of
Theorem \ref{thm:fmpCn}).
That $\Di \top$ is valid is trivial.
The validity of $\ACon$ follows from Lemma \ref{lem:filtr-for-conn} and
Proposition \ref{prop:conn}.
\qed
\end{proof}

\begin{theorem}\label{thm:chrom-conn-compl}
Let $\clC$  be the class of connected non-singleton graphs,
$\clC^{>k}$ the class of non-$k$-colorable graphs in $\clC$.
Then $\Log{\diff{\clC}}=\diff{\LB{}}+\{\ACon,\Di\top\}$,
and $\Log{\diff{\clC}^{>k}}=\diff{\LB{}}+\{\CF{k},\ACon,\Di\top\}$.
\end{theorem}
\begin{proof}
Similar to the proof of Theorem \ref{thm:chrom-compl}.
Completeness of
$\diff{\LB{}}+\{\ACon,\Di\top\}$
and $\diff{\LB{}}+\{\CF{k},\ACon,\Di\top\}$ with respect to their finite point-generated frames
follows from Theorem \ref{thm:fmpCnConn}.

Assume that
$F=(X,R,D)$ is a point-generated $\diff{\LK{B}}$-frame, and $(X,R)$ is connected and validates $\Di \top$.
Consider the repairing $\dunr{F}=(Y,S,\neq_Y)$ of $F$. Clearly, $\Di\top$ is valid in $\dunr{F}$.
Let $(x,i)$ and $(y,j)$ be in $Y$. First, assume that $x\neq y$. Since $(X,R)$ is connected, there is a path between $x$ and $y$ in $(X,R)$, which
induces
a path between $(x,i)$ and $(y,j)$ in $(Y,S)$ by the definition of $S$.
Now consider two distinct points $(x,i)$ and $(x,j)$ in $Y$. Since $\Di\top$ is valid in $F$,
we have $xRy$ for some $y$ in $F$. Then we have $(x,i)S(y,0)$ and $(x,j)S(y,0)$.
It follows that
$(Y,S)$ is connected and so
$\dunr{F}$  validates $\ACon$ by Proposition \ref{prop:conn}.

That other axioms hold in $(Y,S,\neq_Y)$ was shown in Theorem \ref{thm:chrom-compl}.
Now the theorem follows from the fact that $F$ is a p-morphic image of $\dunr{F}$.
\qed
\end{proof}

\section{Corollaries}\label{sec:coroll}

Lemmas \ref{lem:filtr-for-chrom} and \ref{lem:filtr-for-conn} were stated in a more general way than it was required for the proofs of Theorems  \ref{thm:fmpCn} and  \ref{thm:fmpCnConn}.
The aim of using these, more technical, statements is the following.

\begin{definition}
\normalfont
A logic $L$ \emph{admits (rooted) definable filtration},
if for any (point-generated) model $M$ with $M\mo L$,
and for any finite $\SubFrs$-closed set of formulas~$\Gamma$,
there exists a finite model $\ff{M}$ with $\ff{M}\mo L$ that is a definable $\Gamma$-filtration of~$M$.
\end{definition}

In \cite{KSZ:AiML:2014,KikotShapZolAiml2020}, it was shown that if a modal logic $L$ admits  definable filtration,
then its enrichments with modalities for the transitive closure and converse relations also admit
definable filtration.

Notice that if $L=\LK{}_2+\vf$, where $\LK{}_2$ is the smallest bimodal logic and $\vf$ is a bimodal formula,
then $M\mo L$ iff $M\mo \scheme{\vf}$.
In particular, the logics $\LK{}_2+\CF{k}$ admit definable filtration by Lemma \ref{lem:filtr-for-chrom}.
This fact immediately extends to any bimodal logic $L+\CF{k}$, whenever $L$ admits definable filtration.

\begin{corollary}
If a bimodal logic $L$ admits definable filtration, then
all $L+\CF{k}$ admit definable filtration, and consequently have the finite model property.
\end{corollary}

Applying
Lemmas \ref{lem:filtr-for-chrom} and \ref{lem:filtr-for-conn}
to the case of point-generated models, we obtain the following version of Theorems  \ref{thm:fmpCn} and  \ref{thm:fmpCnConn}.
\begin{corollary}\label{cor:adf-weak}
Assume that a bimodal logic
$L$ admits rooted definable filtration, $k<\omega$.
Then $L+\CF{k}$ has the finite model property.
If also $L$ extends $\diff{\LK{B}}$, then $L+\{\CF{k},\ACon\}$ has the finite model property.
\end{corollary}

\later{
Notice that this corollary does not imply the exponential model property.
It does not imply the admits definable filtration property as well.
}

\section{Discussion}\label{sec:concl}
We have shown that modal logics of different classes of non-$k$-colorable graphs are decidable.
It is of definite interest to consider logics of certain graphs, for which the chromatic number is unknown.

Let $F=(\mathbb{R}^2, R_{=1})$ be the unit distance graph of the real plane.
It is a long-standing open problem what is $\C{F}$ (Hadwiger–Nelson problem). It is known that $5\leq \C{F}\leq 7$ \cite{de2018chromatic},\cite{Exoo2020}.\todo{Check later; De2018 - Combinatorics? }

Let $L_{=1}$ be the bimodal logic of the
frame $(\mathbb{R}^2, R_{=1},\neq_{\mathbb{R}^2})$.
In modal terms, the problem asks whether $\CF{5}, \CF{6}$ belong to $L_{=1}$.
We know that $L_{=1}$ extends $L=\diff{\LK{B}}+\{\CF{4}, \ACon, \Di \top, \Di p\imp\Did p\}$ (it is an easy corollary of the above results that $L$ is decidable).
However, $L_{=1}$ contains extra formulas. For example, consider the formulas
$$
\PF(k,m,n)=\bigwedge_{i<k} \Di^m \Box^n p_i \imp \bigvee_{i\neq j<k} \Di^m (p_i \con p_j).
$$
For various $k,m,n$, $\PF(k,m,n)$ is in $L_{=1}$ (and not in $L$); this can be obtained from known solutions for problems of packing equal circles in a circle.

\begin{problem}\label{problem:1}
Is $L_{=1}$ decidable? Finitely axiomatizable? Recursively enumerable? Does it have the finite model property?
\end{problem}

Notice that instead of considering the difference auxiliary modality, one can consider the logic with the universal modality: this logic is a fragment of $L_{=1}$, but still can express formulas  $\CF{k}$.

Let $V_r\subseteq \mathbb{R}^2$ be a disk of radius $r$. It follows from de Bruijn–Erdős theorem,
that if $\C{F}>k$, then $\C{V_r,R_{=1}}>k$ for some $r$.

Let $L_{=1,r}$ be the unimodal logic of the frame $(V_r, R_{=1})$. If $r>1$, then the universal modality is expressible, and so are the formulas $\CF{k}$. Hence, it is of interest to consider axiomatization problems and algorithmic problems for these logics.
\begin{problem}
To analyze the unimodal logics $L_{=1,r}$.
\end{problem}

\section{Acknowledgements}
The author would like to thank the reviewers for their helpful comments on an earlier version of the paper.

%
%
%

\bibliographystyle{alpha}
\bibliography{chromatic}
\end{document}